\documentclass[12pt,cd]{article}
\usepackage{amssymb}
\usepackage{amsfonts}
\usepackage{mathrsfs}
\usepackage{amsmath}
\usepackage{amsthm}
\usepackage{slashed}
\usepackage{url}
\usepackage{bm}
\usepackage{enumitem}
\usepackage{graphicx}
\usepackage{setspace}
\usepackage[compact]{titlesec}
\usepackage{abstract}
\usepackage{lmodern}
\usepackage[T1]{fontenc}
\usepackage{xcolor}
\usepackage{float}
\usepackage[all,cmtip]{xy}
\usepackage{mathtools}
\usepackage{tikz}
\usetikzlibrary{positioning,decorations.pathreplacing,decorations.markings,arrows.meta,calc,fit,quotes,cd,math,arrows,backgrounds,shapes.geometric,cd}
\usepackage[final,hyperindex,linktoc=page,hyperfootnotes=true]{hyperref}

\usepackage[noabbrev,capitalise]{cleveref}
\crefformat{equation}{(#2#1#3)}
\crefformat{enumi}{#2#1#3}
\crefformat{enumii}{#2#1#3}
\usepackage{thmtools}

\renewcommand{\phi}{\varphi}
\renewcommand{\theta}{\vartheta}

\newcommand{\ca}{\mathcal}
\newcommand{\nc}{\mathsf}
\newcommand{\mf}{\mathfrak}
\newcommand{\Set}{\nc{Set}}

\newcommand{\Loc}{\nc{Loc}}

\newcommand{\Frm}{\nc{Frm}}

\newcommand{\slat}{\nc{SLat}} 
\newcommand{\jslat}{\lor\text{-}\slat}

\newcommand{\Pos}{\nc{Pos}}

 \DeclareMathOperator{\Sh}{\mathrm{Sh}} 

\newcommand{\idl}[1]{{\mathrm{idl}}#1}
\newcommand{\id}{\mathrm{id}}

\newcommand\twoheaddownarrow{\mathrel{\rotatebox[origin=c]{270}{$\twoheadrightarrow$}}}

\theoremstyle{plain}
\newtheorem{theorem}{Theorem}[section]
\newtheorem{corollary}[theorem]{Corollary}
\newtheorem{lemma}[theorem]{Lemma}
\newtheorem{proposition}[theorem]{Proposition}

\theoremstyle{remark}
\newtheorem{remark}[theorem]{Remark}
\newtheorem{example}[theorem]{Example}

\theoremstyle{definition}

\author{Vasileios Aravantinos-Sotiropoulos\footnote{National Technical University of Athens, v\_aravantinos@mail.ntua.gr},
	Panagis Karazeris\footnote{Department of Mathematics,
		University of Patras, 26504 Patras, Greece,
		pkarazer@upatras.gr},\\
	and Joshua Wrigley\footnote{Department of Mathematics and Statistics, Faculty of Sciences, Masaryk University, Kotlářská 2, 611 37 Brno, Czech Republic, wrigley@math.muni.cz}}

\usepackage[
backend=biber,
sorting=nyt,
style=ieee
]{biblatex}
\title {Locales in presheaf toposes vs.\ presheaves of locales}
\date{}

\begin{document}
\maketitle
\begin{abstract}
By a well-known characterisation, in a presheaf topos every internal suplattice is a presheaf of suplattices, but not every presheaf of suplattices is an internal suplattice (and similarly for frames).  In this paper, we construct the free internal suplattice/frame on an presheaf of suplattices/frames, yielding a left adjoint to the forgetful functor from the respective internal structures to presheaves of structures. The description of this left adjoint has also appeared in recent work of Henry and Townsend, in connection to a different universal property, namely that of turning a lax natural transformation between poset-enriched functors to a strict one.
As an application of our construction, we 
investigate conditions on frames internal to a presheaf topos, such as being locally compact, compact, stably locally compact or Hausdorff, 
in terms
of properties of their sections in the base topos. In the first three cases, it is necessary that all the sections have the respective properties, while the Hausdorff property is not transferred to the sections.
Moreover for local compactness it is necessary that the transition maps preserve the way-below relation. Finally, for an internal locally compact frame in presheaves we analyse the connection of its way-below relation to the respective relations of its sections.
\end{abstract}

\renewcommand{\thefootnote}{\fnsymbol{footnote}} 
\renewcommand{\thefootnote}{\arabic{footnote}} 

\section{Introduction-Preliminaries}
The aim of this work is to contribute to the study of locales (frames) that are internal to a presheaf topos in terms of their `external' description. By their `external' description, we mean the consequence of the well-known characterisation, due to Joyal and Tierney \cite{jt}, that in the topos of presheaves over a small category (with finite limits), the internal frames are precisely presheaves of frames such that the transition maps have left adjoint maps satisfying the Beck-Chevalley condition (accounting for the internal completeness of the underlying sup-lattice) that also satisfy the Frobenious reciprocity condition (accounting for the distributivity of internal suprema with finite infima). 

In this paper, we show that the forgetful functor from internal sup-lattices (respectively, internal frames) to presheaves of sup-lattices (resp., frames) has a left adjoint, which we denote by $\mf{z}$ (\cref{Thm: z left adjoint for suplattices}, resp., \cref{Thm: z left adjoint for internal frames}). The existence of the left adjoint $\mf{z}$, for the case of frames, is a special case of an internal `generalised ideal construction', as studied in \cite{car} and in the work of the third named author \cite{wr,wrphd} (for the non-internal version, see \cite[Proposition II.2.11]{ss}). The underlying construction of $\mf{z}$ giving the free internal sup-lattice (resp., frame) associated with a presheaf of sup-lattices (resp., frames) has also appeared in the work of Henry and Townsend \cite{ht}, but in connection to a different universal property it enjoys: $\mf{z}$ universally turns lax natural transformations between functors, valued in suitable poset-enriched categories, into (strict) natural transformations (for a comparison of these two universal properties, see \cref{rem:comparison_lax_to_nat}).  The existence of a left adjoint to the forgetful functor from internal suplattices/frames to presheaves of suplattices/frames can be leveraged to give a clear, algebraic account of how to compute, among other things, the internal power-object and the internal ideal-object (\cref{coro:internal-power-object} and \cref{coro:internal_ideals}).  This connection between the free suplattice completion $\mf{z}$ and internal power-objects and ideal-objects is also present in \cite{ht}; in contrast to \cite{ht}, our proof is straightforwardly algebraic. Moreover, we show that the free internal frame completion $\mf{z}$ is idempotent if and only if all presheaves of frames are internal frames, if and only if the presheaf topos is Boolean (\cref{Groupoid characterization}).

Although \cite{ht} provides an equivalence between the category of compact regular frames that are internal in a presheaf topos and the category of presheaves with values in the category of compact regular frames, this equivalence is not obtained taking the underlying presheaf of an internal frame, {\it à la} Joyal--Tierney, and so it remains a meaningful question to investigate compactness and Hausdorffness conditions in terms of properties of the sections of the internal frame.  We use the explicit description of the free frame completion $\mf{z}$ to analyse conditions such as those of being locally compact, compact, stably locally compact, or Hausdorff for locales internal to a presheaf topos. The importance of these respective subclasses of internal locales is, we believe, manifested by the various applications they have found in the literature. For instance, compact Hausdorff locales are crucial in formulating a Gelfand-Neimark duality internally in a presheaf topos \cite{h,bm}, which is then used in \cite{hls,svw} in a discussion of topos-theoretic foundations to algebraic quantum theory. 

The explicit computation of the sections of the internal ideal-object provides the first stepping stone for relating properties of the sections of an internal frame $L$ to internal properties of $L$. For example, compactness of the frame $L$ amounts to the equalizer of the supremum and the constant to top map
\[
\begin{tikzcd}
	\idl L \ar[shift left = 2]{r}{\bigvee} \ar[shift right = 2]{r}[']{1} & L
\end{tikzcd}
\]
being just $\{L\}$. By using $\mf{z}$ to calculate internal ideal-objects, it can easily be derived that the sections of a compact internal frame are compact (\cref{Prop: sections of compact}). Further, local compactness amounts to the existence of a map $\Lambda \colon L \to \idl L,$ left adjoint to the supremum map, and from that we can derive a section-wise left adjoint to the supremum map,
$\lambda_a \colon La \to \idl (La)$ 
yielding the local compactness of the sections of $L$ (\cref{Internal LC implies sectionwise LC}), and it can further be calculated that the transition maps preserve the way-below relation of the sections (\cref{Internal LC implies preservation of way-below}), or equivalently the localic transition maps are finitary. However, we show that the internal way-below relation of an internal locally compact frame is not reducible to the way-below relations of its sections.
In a similar fashion, stable local compactness (preservation of finite infima by the extra left adjoint) can be transferred to the sections of the internal locale (\cref{Prop: sections of stably locally compact}).
In connection to the Hausdorff condition, i.e the closedness of the localic diagonal
$L \to L\times L$, we obtain that, for a presheaf $M$ of Hausdorff frames, the associated internal frame $\mf{z}M$ is internally Hausdorff (\cref{Prop: sections Hausdorff implies zL Hausdorff}). In contrast, a counterexample, suggested to us by S.~Henry, shows that the Hausdorffness of an internal frame need not be transferred to the sections.

\paragraph{Terminology, notation and conventions.} 
We now fix some terminology and notation that will be used hereafter. 

Firstly, given an arrow $f $ in a category $\ca{C}$, we use $\vartheta_0(f)$ to denote the source of $f$, and $\vartheta_1(f)$ for the target.  Often we will work with objects internal to the presheaf topos $[\ca{C}^{op},\Set]$, where we adopt the convention that $\ca{C}$ has at least (weak) pullbacks.  Instances at an object $a$ of a natural transformation $\varphi$ may be denoted either as $\varphi_a$ or as $\varphi^a$, especially when a map $ \varphi^a_*$ adjoint to the latter occurs. 

A \emph{frame} is a complete lattice $L$ satisfying the infinite distributive law
\[
\textstyle u\wedge \bigvee_{i\in I}{v_i} = \bigvee_{i\in I}{(u\wedge v_i)}
\]
for all $u \in L$ and subsets $\{v_i\mid i\in I\}\subseteq L$. A \emph{frame homomorphism} is a map $L\to M$ between frames that preserves finite infima and arbitrary suprema (in particular, frame homomorphisms preserve $0$ and $1$). The category of frames and their homomorphisms is denoted by $\Frm$, and is complete and cocomplete (\cite[\S IV.3\&4]{picado_pultr}).

The category $\Loc$ is defined as the opposite of $\Frm$; the objects of $\Loc$ are referred to as \emph{locales} and its morphisms as \emph{localic maps}. 
We usually do not distinguish notationally between an object $L$ in the category of locales and the same object in the category of frames (although it is customary when we denote by $X$ the former, to write $OX$ to denote the same object, this time regarded as a frame). Accordingly, the frame homomorphism corresponding to a localic map $f \colon L \to M$ is denoted by $f$ while we write $f_*$ for its right adjoint (direct image of the localic map), which exists
since $f$ preserves (finite infima and) all suprema.

An ideal of a lattice $L$ is a subset $I \subseteq L$ such that 
\[0\in I, \qquad
x\leq y \in I \Rightarrow x\in I, \quad \text{and} \quad
x, \; y \in I \Rightarrow x\vee y \in I.\]
We denote by $\idl L$ the poset of ideals, ordered by inclusion,  which is a frame when $L$ is a distributive lattice. A morphism of lattices $f \colon L \to M$
induces a pair if adjoint maps $f[-] \dashv f^{-1}[-] \colon \idl M \to \idl L .$ Here $f^{-1}[-]$ denotes the inverse image of $I,$ which is 
an ideal when $I$ is, and $f[I]$ is \textit{the ideal generated} by the set-theoretic direct image (we adopt this convention to make our notation less cumbersome). When needed we denote by 
$$\langle S \rangle =\{ x \in L \; |\; \exists s_1 , \cdots , s_n \in S  \; x\leq s_1 \vee \cdots \vee s_n \}$$ 
the ideal generated by   $S\subseteq L.$ In particular 
$\langle \bigcup _i J_i \rangle  $ 
is the ideal generated by a union of ideals. When $L,$ $M$ are distributive lattices, $f^{-1}[-]$ is a frame morphism.

The way-below relation $\ll$ on a poset with directed suprema is defined as
$x\ll y  $ if, whenever $y\leq \bigvee S,$ with $S$ directed, there is an $s\in S$ so that  $x\leq s.$ A complete lattice is continuous if for all $x,$ the set 
$\twoheaddownarrow \! x =\{y \; |\; y\ll x\}$ is directed and $x=\bigvee \twoheaddownarrow \! x.$ 
This can be equivalently expressed (\cite{ss}, VII 2.1) as requiring that the
supremum map $\idl L \to L$ has a left adjoint ($x\mapsto \twoheaddownarrow \! x $) 
A frame is locally compact if it is a continuous lattice.

The concepts and results that we recalled above make sense and remain valid in the internal logic of a topos. As we shall see, concepts such as
local compactness and compactness will be defined internally in a topos via the poset of ideals.

\section{Adjoining internal suprema}\label{sec:free_const}
Frames that are internal in a presheaf topos $[\ca{C}^{op}, \Set]$ are presheaves of frames satisfying further conditions (\cite[\S VI.4]{jt}, \cite[\S C1.6.9]{el}, where we recall our convention that $\ca{C}$ has all pullbacks). 
In particular, the internal completeness of the poset $L$ is captured by the existence, for each $f\colon b\to a \in  \ca{C},$
of a left adjoint $\Sigma _f \dashv Lf \colon La \to Lb,$ such that a Beck-Chevalley condition is satisfied for pullbacks in $\ca{C}.$
A morphism of internal sup-lattices $\varphi \colon L \to M$ has to preserve suprema indexed by objects in the topos, so that all squares
\begin{displaymath}
	\xy
	\xymatrix{ Lb \ar[d]_{\Sigma ^L _f } \ar[r]^{ \varphi ^{b }} & Mb \ar[d]^{\Sigma ^M _f}\\
		La \ar[r]^{ \varphi ^a}	& Ma}
	\endxy
\end{displaymath}
must commute. If $L$  is moreover a frame, then $\Sigma _f \dashv Lf$ is further required to satisfy the Frobenious reciprocity condition $\Sigma_f( Lf(u) \land v) = u \land \Sigma_f v$ for all $u \in La$ and $v \in Lb$ (so that $L$ as a functor takes values in the category of frames and open maps). We remark here that, as it is not too hard to see, these characterisations in fact remain valid if $\ca{C}$ only has \emph{weak} pullbacks and the Beck-Chevalley condition is required for the latter. (For cases where $\ca{C}$ lacks even weak pullbacks, the reader is directed to \cite{car,wr_tac}.)

As a result, there are forgetful functors
\begin{align*}
	U \colon \nc{SupLat}[\ca{C}^{op}, \Set] & \longrightarrow [\ca{C}^{op}, \nc{SupLat} ], \\
	U \colon \Frm[\ca{C}^{op}, \Set] & \longrightarrow [\ca{C}^{op}, \Frm ],
\end{align*}
which are however not full. Indeed, not every natural transformation between suplattice or frame-valued functors coming from internal suplattices or frames respectively need preserve suprema in the internal sense described above.  As aforementioned, as a particular case of the ``geometrisation'' procedure described in \cite{car,wr,wrphd}, we can construct left adjoints to these forgetful functors.  In this section, we provide a complete description of the left adjoint, as well as explicit proofs of its universal property.
\begin{remark}
	Recall from \cite{lawvere} that the left adjoints $\Sigma_f$, satisfying Beck-Chevalley and Frobenius, are used in categorical logic to interpret existential quantification.  From the logical perspective, a presheaf of frames can be understood as encoding the algebraic syntax of some theory of geometric logic \emph{without} existential quantification (see \cite[\S D1]{el} for more on the syntax of geometric logic).  Our free internal frame on a presheaf of frames is therefore similar in spirit to the existential completions found in \cite{trotta,wr_ex_comp}, a connection that will be made concrete in \cref{rem:comparision_excomp}.
\end{remark}

\subsection{The free internal suplattice on a presheaf of suplattices.}
For suplattices, a functor $\mf{z}$ in the opposite direction is defined as follows: 
$$\mathfrak{z}L (a) = \left\{(u_f )\in \prod _{\vartheta _1 f=a}\; L(\vartheta _0 f) \; \middle\vert \; \forall g\colon \vartheta _0 (g) \to \vartheta _0 (f) \;( L(g)(u_f ) \leq u_{fg}) \right\}.$$
For an arrow $r \colon b\to a$ in $\mathcal{C},$ $\mathfrak{z}L(r) \colon  \mathfrak{z}L(a) \to \mathfrak{z}L(b)$ sends the tuple 
$(u_f )_{\vartheta _1 f=a}$ to $(u_{rh})_{\vartheta _1 h=b}$. Each $\mathfrak{z}L (a)$ is then a complete lattice with suprema given componentwise, as it is easy to verify. Since $L$ takes values in $\nc{SupLat}$, a left adjoint $\Sigma ^{\mathfrak{z}L}_r $ to $\mf{z}Lr$ can  be defined as follows
\[
\textstyle\Sigma_{r}^{\mathfrak{z}L}(u_f)_{\vartheta_1f=b} =  \left(\bigvee _{rf=g} u_f \right)_{\vartheta_1g=a} ,
\]
which we can easily confirm yields a left adjoint:
\begin{eqnarray*}
	\textstyle\left(\bigvee _{rf=g} u_f \right)_{\vartheta_1 g = a} \leq (v_g)_{\vartheta_1 g = a} &\Leftrightarrow & \textstyle\forall g \; \bigvee _{rf=g} u_f  \leq v_g \\
	&\Leftrightarrow &\forall g \; \forall f \; (rf = g \Rightarrow  u_f \leq v_g \; )  \\
	&\Leftrightarrow & \forall f \; (u_f \leq v_{rf} )  \\
	&\Leftrightarrow & (u_f)_{\vartheta_1 f = b} \leq \mf{z}L(r)(v_{rf})_{\vartheta_1 f =b}.
\end{eqnarray*}

Next, we check the Beck-Chevalley condition, i.e.\ for a pullback square
\begin{displaymath}
	\begin{tikzcd}
		d\ar[r,"q"]\ar[d,"p"'] & c\ar[d,"s"] \\
		b\ar[r,"r"] & a,
	\end{tikzcd}
\end{displaymath}
the following square also commutes
\begin{displaymath}
	\begin{tikzcd}
		\mf{z}L(b)\ar[r,"\Sigma_{r}^{\mathfrak{z}L}"]\ar[d,"\mf{z}L(p)"'] & \mf{z}L(a)\ar[d,"\mf{z}L(s)"] \\
		\mf{z}L(d)\ar[r,"\Sigma_{q}^{\mathfrak{z}L}"] & \mf{z}L(c).
	\end{tikzcd}
\end{displaymath}
For any $(u_g)_{\vartheta_1g=b}\in\mf{z}L(b)$ we have that
\[\textstyle\mf{z}L(s)\Sigma_{r}^{\mf{z}L}(u_g)_{\vartheta_1g=b}=\mf{z}L(s)\left(\bigvee_{rg=f}u_g\right)_{\vartheta_1f=a}=\left(\bigvee_{rg=sh}u_g\right)_{\vartheta_1h=c},\]
\[\textstyle\text{and} \quad \Sigma_{q}^{\mf{z}L}\mf{z}L(p)(u_g)_{\vartheta_1g=b}=\Sigma_{q}^{\mf{z}L}(u_{pk})_{\vartheta_1k=d}=\left(\bigvee_{qk=h}u_{pk}\right)_{\vartheta_1h=c}.\]
Since $rg=sh$ happens if and only if there is a $k$ such that $g=pk$ and $qk=h$, it is clear that for each $h$ with $\vartheta_1h=c$ the corresponding components are equal.

A natural transformation $\varphi \colon L \to M$ in $[\ca{C}^{op}, \nc{SupLat} ]$ induces a morphism
$\mf{z} \varphi \colon \mf{z}L \to \mf{z}M$ with 
$\mf{z} \varphi ^{a} \colon \mf{z}L(a) \to \mf{z}M(a)$ acting by
\[
\mf{z} \varphi ^{a} ((u_f ) _{\vartheta _1 f=a})= (\varphi ^{\vartheta _0 f}(u_f ))_{\vartheta _1 f=a} .
\]
This is well-defined, in the sense that $\mf{z} \varphi ^{a} ((u_f ) _{\vartheta _1 f=a}) \in (\mf{z}M)(a)$, since, for $g \colon \vartheta _0 g \to \vartheta _1 g = \vartheta _0 f$, 
\[Mg (\varphi ^{\vartheta _0 f}(u_f ))=\varphi ^{\vartheta _0 (fg)} Lg (u_f ) \leq \varphi ^{\vartheta _0 (fg)} (u_{fg}) ,\]
by the commutativity of the following naturality square (and noticing that $\vartheta _0 (fg) = \vartheta _0 (g)$)
\begin{displaymath}
	\xy \xymatrix{ L(\vartheta _0 f) \ar[d]_{Lg} \ar[r]^{ \varphi ^{\vartheta _0 (f)} } & M(\vartheta _0 f) \ar[d]^{Mg}\\
		L(\vartheta _0 g) \ar[r]^{ \varphi ^{\vartheta _0 (g)} }	& M(\vartheta _0 g)	.}
	\endxy
\end{displaymath}
In addition, it is easy to see that $\mf{z}\varphi$ is natural, and moreover it is clear that each $\mf{z}\varphi^{a}$ preserves suprema, since these are computed componentwise in $\mf{z}L(a)$ and each instance of $\varphi$ is a morphism in $\nc{SupLat}$. Finally, $\mf{z}\varphi$ also preserves suprema in the internal sense, i.e.\ for every $r\colon b\to a\in\ca{C}$ the following square commutes,
\begin{displaymath}
	\begin{tikzcd}
		\mf{z}L(b)\ar[r,"\mf{z}\varphi^{b}"]\ar[d,"\Sigma^{\mf{z}L}_{r}"'] & \mf{z}M(b)\ar[d,"\Sigma^{\mf{z}M}_{r}"] \\
		\mf{z}L(a)\ar[r,"\mf{z}\varphi^{a}"] & \mf{z}M(a),
	\end{tikzcd}
\end{displaymath}
which can be seen from the following calculation:
\begin{align*}
	\mf{z}\varphi^{a}\Sigma_{r}^{\mf{z}L}(u_g)_{\vartheta_1g=b} &= \textstyle\mf{z}\varphi^{a}\left(\left(\bigvee_{rg=h}u_g\right)_{\vartheta_1h=a}\right) \\
	& = \textstyle\left(\varphi^{\vartheta_0h}\left(\bigvee_{rg=h}u_g\right)\right)_{\vartheta_1h=a} \\
	&= \textstyle\left(\bigvee_{rg=h}\varphi^{\vartheta_0h}(u_g)\right)_{\vartheta_1h=a} \\
	& =\Sigma_{r}^{\mf{z}M}(\varphi^{\vartheta_0g}(u_g))_{\vartheta_1g=b} \\
	&=\Sigma_{r}^{\mf{z}M}\mf{z}\varphi^{b}(u_g)_{\vartheta_1g=b}.
\end{align*}
\begin{remark}		
	Being a suplattice morphism, each $\mathfrak{z} \varphi ^{a} \colon \mathfrak{z}L(a) \to \mathfrak{z}M(a)$ has a right adjoint, which we will denote by
	$\mathfrak{z} \varphi _* ^{a} \colon \mathfrak{z}M(a) \to \mathfrak{z}L(a).$ Since we will need its explicit description later, let us record here that this is given by
	$$\mathfrak{z} \varphi _* ^{a} ((v_f )_{\vartheta _1 f =a}) =  (\varphi _* ^{\vartheta_0f} (v_f )_{\vartheta _1 f =a}),$$
	where, for every $a\in\ca{C}$, $\varphi _* ^{a}$ denotes the right adjoint of the suplattice morphism $\varphi ^{a} \colon La \to Ma $.
	To see that this is a correct definition, observe that for the maps displayed in the diagram
	\[\begin{tikzcd}
		{L(\theta_0f) } & {M(\theta_0 f)} \\
		{L(\theta_0g)} & {M(\theta_0 g)},
		\arrow["{\phi^{\theta_0 f}}", shift left=2, from=1-1, to=1-2]
		\arrow["Lg"', from=1-1, to=2-1]
		\arrow["{\phi_\ast^{\theta_0 f}}", shift left=2, from=1-2, to=1-1]
		\arrow["Mg", from=1-2, to=2-2]
		\arrow["{\phi^{\theta_0 g}}", shift left=2, from=2-1, to=2-2]
		\arrow["{\phi_\ast^{\theta_0 g}}", shift left=2, from=2-2, to=2-1]
	\end{tikzcd}\]
	we have that
	\[\varphi ^{\vartheta _0 g}\;  Lg \; \varphi _* ^{\vartheta _0 f}(v_f ) =
		Mg \; \varphi ^{\vartheta _0 f} \; \varphi _* ^{\vartheta _0 f} (v_f ) \leq Mg (v_f ) \leq v_{fg}, \]
	which is equivalent to 
	$Lg \; \varphi _* ^{\vartheta _0 f}(v_f ) \leq \varphi _* ^{\vartheta _0 (fg)}(v_{fg})$, as required.
	The adjointness $\mathfrak{z} \varphi ^{a} \dashv \mathfrak{z} \varphi _* ^{a}$ follows by the fact that the order on $\mf{z}L(a)$ and $\mf{z}M(a)$ is given pointwise.
\end{remark}

Summing up our observations so far, we have shown that $\mf{z}$ defines a functor $[\ca{C}^{op},\nc{SupLat}]\to\nc{SupLat}[\ca{C}^{op},\nc{Set}]$ (the functoriality of $\mf{z}$ was also shown in \cite[Appendix]{towncah}).

\begin{theorem}\label{Thm: z left adjoint for suplattices}
	The functor $\mf{z}\colon[\ca{C}^{op},\nc{SupLat}]\to\nc{SupLat}[\ca{C}^{op},\nc{Set}]$ is left adjoint to the forgetful functor $U\colon \nc{SupLat}[\ca{C}^{op},\nc{Set}]\to[\ca{C}^{op},\nc{SupLat}]$.
\end{theorem} 
\begin{proof} 
	The unit of the adjunction will be the morphism $\eta_{L}\colon L\to \mf{z}L$ defined for each $a\in\ca{C}$ and $x\in La$ by $\eta_{L}^{a}(x)\coloneqq (Lf(x))_{\vartheta_1f=a}$. Note that this is well-defined and natural in $a$ simply by functoriality of $L$, while it preserves suprema because the latter are computed componentwise in $\mf{z}L(a)$ and the transition maps $Lf$ are morphisms in $\nc{SupLat}$. In addition, for any natural transformation of presheaves of suplattices $\varphi\colon L\to M\in[\ca{C}^{op},\nc{SupLat}]$ and $a\in\ca{C}$, $x\in La$ we have that
	\begin{align*}
		\mf{z}\varphi^{a}(\eta_{L}^{a}(x))&=\mf{z}\varphi^{a}(Lf(x))_{\vartheta_1f=a} \\
		&=(\varphi^{\vartheta_0f}(Lf(x)))_{\vartheta_1f=a} \\
		& =(Mf(\varphi^{a}(x)))_{\vartheta_1f=a} \\
		&=\eta_{M}^{a}(\varphi^{a}(x)),
	\end{align*}
	showing that $\eta$ is natural in $L$.
	
	Now consider a morphism $\varphi\colon L\to M\in[\ca{C}^{op},\nc{SupLat}]$ with $M\in\nc{SupLat}[\ca{C}^{op},\nc{Set}]$ an internal suplattice. We want to show that $\phi$ extends uniquely along $\eta_L$ to a morphism $\bar{\varphi}\colon\mf{z}L\to M$ in $\nc{SupLat}[\ca{C}^{op},\nc{Set}]$. For each $a \in \mathcal{C}$, define 
	\[
	\textstyle\bar{ \varphi}^a(u_f)_{\vartheta_1f=a} =  \bigvee_{\theta_1 f =a} \Sigma^M_f \varphi ^{\vartheta_0 f} (u_f) .
	\]
	Concerning the naturality of $\bar{ \varphi}$ we have to show that, for $r\colon b \to a,$ the square
	\begin{displaymath}
		\xy \xymatrix{ \mathfrak{z}La \ar[d]_{\mathfrak{z}Lr } \ar[r]^{\bar{ \varphi} ^a} & Ma \ar[d]^{Mr}\\
			\mathfrak{z}Lb \ar[r]^{ \bar{ \varphi} ^b}	& Mb	}
		\endxy
	\end{displaymath} commutes.  We have that	
	\begin{equation}\label{eq:bar_phi_naturality_1}
		\textstyle \bar{ \varphi}^b  \mathfrak{z}Lr (u_f)_{\vartheta_1f=a} = \bar{\varphi}^b  ( u_{rg} )_{\vartheta_1g=b} = \bigvee_{\vartheta_1g=b} \Sigma ^M _g \varphi^{\vartheta _0 g} (u_{rg}),
	\end{equation}
	while 
	\begin{equation}\label{eq:bar_phi_naturality_2}
		\textstyle Mr \bar{ \varphi}^a (u_f )=Mr\left(\bigvee_{\vartheta_1f=a} \Sigma ^M_f \varphi^{\vartheta _0 f}(u_f)\right)= \bigvee_{\vartheta_1f=a} Mr\Sigma ^M_f \varphi^{\vartheta _0 f}(u_f).
	\end{equation}
	Since, for all $g$ with $\vartheta_1g=b$,
	$$\Sigma_g^M \varphi^{\vartheta _0 g}(u_{rg}) \leq Mr \Sigma^M _r \Sigma^M _g\varphi^{\vartheta _0 g}(u_{rg})= Mr \Sigma^M_{rg}\varphi^{\vartheta _0 g}(u_{rg}),$$ 
	and the composites $rg$ are among the arrows $f$ with codomain $a$, it follows that \cref{eq:bar_phi_naturality_1} is less than or equal to \cref{eq:bar_phi_naturality_2}. For the converse inequality, given any $f$ with $\vartheta_1f=a$, we apply the Beck-Chevalley condition for $M$ with respect to the pullback 
	\begin{displaymath}
		\xy \xymatrix{ \cdot \ar[d]_{k } \ar[r]^{ g} & b \ar[d]^{r}\\
			\cdot \ar[r]^{ f}	& a.	}
		\endxy
	\end{displaymath}
	to observe that $Mr\Sigma ^M_f \varphi^{\vartheta _0 f}(u_f)=\Sigma_g^M Mk\varphi^{\vartheta_0f}(u_f)=\Sigma_g^M\varphi^{\vartheta_0g}Lk(u_f)$.
	Since $u_{rg}=u_{fk}$ and $Lku_{f}\leq u_{fk}$, we have that \cref{eq:bar_phi_naturality_2} is also less or equal to \cref{eq:bar_phi_naturality_1}, so that the two are equal as desired.
	
	It is easy to see that each $\bar{\varphi}^{a}$ preserves suprema, since the $\Sigma_f^M$ and $\varphi^{\vartheta_0f}$ do so and because suprema in $\mf{z}L(a)$ are computed componentwise. The morphism $\bar{\varphi}\colon\mf{z}L\to M$ now also preserves suprema internally, i.e. for every $r\colon b\to a\in\ca{C}$ the following square commutes
	\begin{displaymath}
		\begin{tikzcd}
			\mf{z}L(b)\ar[r,"\Sigma^{\mf{z}L}_r"]\ar[d,"\bar{\varphi}^b"'] & \mf{z}L(a)\ar[d,"\bar{\varphi}^a"] \\
			Mb\ar[r,"\Sigma_r^M"] & Ma
		\end{tikzcd}
	\end{displaymath}
	Indeed, for any $(u_g)_{\vartheta_1g=b}\in\mf{z}L(b)$ we have that
	\begin{align*}
		\bar{\varphi}^a\Sigma_r^{\mf{z}L}(u_g)_{\vartheta_1g=b} &= \textstyle\bar{\varphi}^a\left(\bigvee_{rg=f}u_g\right)_{\vartheta_1f=a} \\ 
		& =\textstyle\bigvee_{\vartheta_1f=a}\Sigma_f^M\varphi^{\vartheta_0f}\left(\bigvee_{rg=f}u_g\right) \\ &=\textstyle\bigvee_{\vartheta_1f=a}\bigvee_{rg=f}\Sigma_f^M\varphi^{\vartheta_0g}(u_g) \\ 
		& =\textstyle\bigvee_{\vartheta_1f=a}\bigvee_{rg=f}\Sigma_r^M\Sigma_g^M\varphi^{\vartheta_0g}(u_g) \\
		&=\textstyle\Sigma_r^M\left(\bigvee_{\vartheta_1f=a}\bigvee_{rg=f}\Sigma_g^M\varphi^{\vartheta_0g}(u_g)\right)\\
		& =\textstyle\Sigma_r^M\left(\bigvee_{\vartheta_1g=b}\Sigma_g^M\varphi^{\vartheta_0g}(u_g)\right) =\Sigma_r^M\bar{\varphi}^b(u_g)_{\vartheta_1g=b}.
	\end{align*}
	
	This natural transformation $\bar{\varphi}$ extends $\varphi$ along $\eta$ because, for all $a$, 
	\[
	\textstyle\bar{\varphi}^a \eta_L ^a (u) = \bigvee_{\vartheta_1f=a}\Sigma_f^M\varphi^{\vartheta_0f}(Lf(u)) =\bigvee_{\vartheta_1f=a}\Sigma_f^M Mf(\varphi^a(u))
	\]
	and $\Sigma _f^M Mf(\varphi^a(u)) \leq \varphi^a(u)$, so that this supremum is less or equal to 
	$\varphi^a (u)$, while $\varphi^a(u)$ itself participates in the family over which the supremum is taken via $f=\id_a$, so that the two are equal. 
	
	Finally, let $\psi\colon\mf{z}L\to M\in\nc{SupLat}[\ca{C}^{op},\nc{Set}]$ be such that $\psi\eta_L=\varphi$. We will show that $\psi=\bar{\varphi}$. Consider any element $(u_f)_{\vartheta_1f=a}$ of $\mf{z}L(a)$. We begin by observing that 
	\begin{align*}
		\textstyle\bigvee_{\vartheta_1g=a}\Sigma_g^{\mf{z}L}\eta_L^{\vartheta_{0}g}(u_g) &=\textstyle\bigvee_{\vartheta_1g=a}\Sigma_g^{\mf{z}L}(Lh(u_g))_{\vartheta_1h=\vartheta_0g} \\ 
		& =\textstyle\bigvee_{\vartheta_1g=a}\left(\bigvee_{gh=f}Lh(u_g)\right)_{\vartheta_1f=a} \\
		&= \textstyle\left(\bigvee_{\vartheta_1g=a}\bigvee_{gh=f}Lh(u_g)\right)_{\vartheta_1f=a} =(u_f)_{\vartheta_1f=a}
	\end{align*} 
	because for every $f,g$ with $\vartheta_1f=\vartheta_1g=a$ and $h$ with $gh=f$ we have $Lh(u_g)\leq u_{gh}=u_f$, while $u_f$ participates in the supremum upon taking $g=f$ and $h=\mathrm{id}_{\vartheta_0f}$. Using this equality, we now have
	\begin{align*}
		\psi^a(u_f)_{\vartheta_1f=a} &= \textstyle\psi^a\left(\bigvee_{\vartheta_1g=a}\Sigma_g^{\mf{z}L}\eta_L^{\vartheta_0g}(u_g)\right) \\
		& =\textstyle\bigvee_{\vartheta_1g=a}\psi^a\Sigma_g^{\mf{z}L}\eta_L^{\vartheta_0g}(u_g) \\ &=\textstyle\bigvee_{\vartheta_1g=a}\Sigma_g^M\psi^{\vartheta_0g}\eta_L^{\vartheta_0g}(u_g) \\ 
		& =\textstyle\bigvee_{\vartheta_1g=a}\Sigma_g^M\phi^{\vartheta_0g}(u_g)
	\end{align*}
	so that indeed $\psi=\bar{\varphi}$.
\end{proof}
From the proof above, it is easy to see that the co-unit $\varepsilon \colon \mathfrak{z}L \to L$ of the adjunction is the suplattice morphism given, for each $a \in \mathcal{C}$, by the formula
\[
\textstyle\varepsilon_L^a (u_f )_{\vartheta_1f=a} = \bigvee_{\theta_1 f = a} \Sigma _f u_f .
\]
On the other hand, given an internal suplattice $L\in[\ca{C}^{op},\Set]$, the unit $\eta_L\colon L\to\mf{z}L$ will of course not be a suplattice morphism in general. 
\begin{remark}\label{rem:comparison_lax_to_nat}
	The construction giving this left adjoint has appeared independently in \cite{ht} 
	applied to presheaves with values in suitable $\nc{Pos}$-enriched categories,
	in connection with another universal property it enjoys: turning lax transformations between 
	$\nc{Pos}$-enriched functors universally into natural transformations.  But why does the same construction $\mf{z}$ satisfy two seemingly unconnected universal properties?  We briefly give some justification as to why we might expect $\mf{z}$ to satisfy both the universal `lax-to-natural' property identified in \cite{ht} as well as the universal property asserted in \cref{Thm: z left adjoint for suplattices}.  
	
	Let $L \in [\ca{C}^{op},\nc{SupLat}]$ be a presheaf of suplattices and $M \in \nc{SupLat}[\ca{C}^{op},\nc{Set}]$ an internal suplattice.  Given a natural transformation $\phi \colon L \to M$, for each arrow $r \colon b \to a$, by taking the \emph{mate} of the commuting square on the left, we obtain the \emph{lax} commuting square on the right:
	\[
	\begin{tikzcd}
		La \ar{r}{\phi^a} \ar{d}[']{Lr} & Ma \ar{d}{Mr} \\
		Lb \ar{r}[']{\phi^b} & Mb
	\end{tikzcd}
	\qquad
	\begin{tikzcd}
		La \ar{d}[']{Lr} \ar[phantom]{rd}[description]{\leq} & \ar{l}[']{\phi^a_*} Ma \ar{d}{Mr} \\
		Lb & \ar{l}{\phi^b_*} Mb,
	\end{tikzcd}
	\]
	by which we mean that there is an inequality $Lr \phi^a_* \leq \phi^b_* Mr$.  Thus, we obtain a \emph{lax natural transformation} $\phi_* \colon M \rightsquigarrow L$ in the sense of \cite[\S 3]{ht}.  Recall that each natural transformation $\phi $ uniquely extends to a natural transformation $\bar{\phi} \colon \mf{z}L \to M$ that moreover commutes with internal suprema in the sense that the square below on the left commutes, or equivalently, by taking right adjoints, that the square on the right commutes:
	\[
	\begin{tikzcd}
		\mf{z} La  \ar{r}{\bar{\phi}^a} & Ma  \\
		\mf{z} Lb \ar{u}{\Sigma^{\mf{z}L}_r} \ar{r}[']{\bar{\phi}^b} & Mb \ar{u}[']{\Sigma^M_r}
	\end{tikzcd}
	\qquad
	\begin{tikzcd}
		\mf{z} La \ar{d}[']{\mf{z}Lr} & \ar{l}[']{\bar{\phi}^a_*} Ma \ar{d}{Mr} \\
		\mf{z} Lb & \ar{l}{\bar{\phi}^b_*} Mb.
	\end{tikzcd}
	\]
	Thus, we have transported the lax natural transformation $\phi_* \colon M \rightsquigarrow L$ to the natural transformation $\bar{\phi}_* \colon M \to \mf{z} L$.  This describes the action on $\phi_*$ under the bijection between lax natural transformations $M \rightsquigarrow L$ and natural transformations $M \to \mf{z} L$ established in \cite{ht}.
\end{remark}
In the sequel we shall use extensively the construction $\mf{z}$, guided by the important observation in \cite[Example 4.2]{ht} that, given an internal poset in presheaves, equivalently a functor $Q \colon \mathcal{C}^{op} \to \nc{Pos}$, then the frame of internal ideals of $Q$ in the presheaf category $\hat{\mathcal{C}}$ has sections given, for $a \in \mathcal{C},$ as
\[\idl_{\hat{\mathcal{C}}}(Q)a = \mathfrak{z} (\idl(Qa)).\]
The same construction applies to give the internal power-object (\textit{ibid}., Example 4.1),
the internal object of down-closed objects etc., as we now explain, giving a more streamlined argument compared to \cite{ht}.
\begin{corollary}\label{coro:internal-power-object}
	The internal power-object $P_{\hat{\ca{C}}}X$ of a presheaf $X\in [\ca{C}^{op},\nc{Set}] = \hat{\ca{C}}$ is given by $\mathfrak{z} (P\circ X) \in [\ca{C}^{op},\nc{SupLat}]$, where $P$ denotes the (covariant) power-set functor.
\end{corollary}
\begin{proof}
	The adjunction $ P \dashv u \colon \nc{SupLat} \to \nc{Set}$, where $u \colon \nc{SupLat} \to \nc{Set}$ is the forgetful functor,
	induces another adjunction 	
	$$P \circ - \dashv u \circ - \colon [\ca{C}^{op},\nc{SupLat}] \to [\ca{C}^{op},\nc{Set}], $$
	which when composed with the adjunction 
	$$\mathfrak{z} \dashv U\colon \nc{SupLat}[\ca{C}^{op},\nc{Set}]\to[\ca{C}^{op},\nc{SupLat}]$$
	from \cref{Thm: z left adjoint for suplattices}
	gives us that the left adjoint to the forgetful functor
	$\nc{SupLat}[\ca{C}^{op},\nc{Set}]\to  [\ca{C}^{op},\nc{Set}]$, in other words the internal power-object, is $\mathfrak{z} (P\circ -)$. 
\end{proof}
\begin{example}\label{ex:comuting_subobj_class}
	In particular, taking $\mathbf{2} \colon \ca{C}^{op} \to \Frm$ to be the constant presheaf on the two-element frame $\{0 \leq 1\}$, it follows that $\mf{z}  \mathbf{2}$ is the subobject classifier $\Omega$ of $[\ca{C}^{op},\nc{Set}]$.  Indeed, for each object $a \in \ca{C}$, an isomorphism between $\mf{z} \mathbf{2}(a)$ and the usual description of $\Omega (a)$ in terms of sieves on the object $a$ (as found in, for instance, \cite[\S I.4]{sgl}) can be witnessed by sending $(u_f)_{\theta_1 f = a} \in \mf{z} \mathbf{2}(a)$ to the sieve $\{f \colon b \to a \mid u_f = 1\} \in \Omega (a)$.
\end{example}
\begin{corollary}\label{coro:internal-down-closed-object}
	The internal object of down-closed-subobjects $D_{\hat{\ca{C}}}X$ of an internal poset $X$ in $ [\ca{C}^{op},\nc{Set}]$ 
	is given by $\mathfrak{z} (D\circ X) \in [\ca{C}^{op},\nc{SupLat}].$ 
\end{corollary}
\begin{proof}
	Similar to \cref{coro:internal-power-object}, taking into account the equivalence
	$$\nc{Pos}[\ca{C}^{op},\nc{Set}] \simeq [\ca{C}^{op},\nc{Pos}]$$
	and the adjunction 
	$$D \dashv u \colon \nc{SupLat} \to \nc{Pos}$$
	between sup-lattices and posets. 
\end{proof}

\subsection{The free internal frame on a presheaf of frames.}
Next, we want to show that the adjunction established above restricts to one between frame-valued presheaves and internal frames in $[\ca{C}^{op},\nc{Set}]$.

First of all, if $L\in[\ca{C}^{op},\Frm]$, then it is easy to see that for every $a\in\ca{C}$ the suplattice $\mf{z}L(a)$ also has finite infima. In particular, given elements $(u_f)_{\vartheta_1f=a}$ and $(v_f)_{\vartheta_1f=a}$, the family $(u_f\wedge v_f)_{\vartheta_1f=a}$ is also an element of $\mf{z}L(a)$ because for any $g$ with $\vartheta_1g=\vartheta_0f$ we have $Lg(u_f\wedge v_f)\leq Lg(u_f)\wedge Lg(v_f)\leq u_{fg}\wedge v_{fg}$. Since both finite infima and arbitrary suprema are given componentwise, it immediately follows that $\mf{z}L(a)$ is a frame.

To check Frobenious reciprocity, consider an arrow $r\colon b\to a\in\ca{C}$ and elements $(u_f)_{\vartheta_1f=a}$ and $(v_g)_{\vartheta_1g=b}$ of $\mf{z}L(a)$ and $\mf{z}L(b)$ respectively. Then using the infinite distributive law which holds in the sections of $L$, we have that
\begin{align*}
	\Sigma_r^{\mf{z}L}\left(\mf{z}Lr(u_f)_{\vartheta_1f=a}\wedge (v_g)_{\vartheta_1g=b}\right) &= \Sigma_r^{\mf{z}L}\left((u_{rg})_{\vartheta_1g=b}\wedge (v_g)_{\vartheta_1g=b}\right) \\ 
	&=\Sigma_r^{\mf{z}L}(u_{rg}\wedge v_g)_{\vartheta_1g=b} \\
	&= \textstyle\left(\bigvee_{rg=f}(u_{rg}\wedge v_g)\right)_{\vartheta_1f=a} \\ 
	&=\textstyle\left(\bigvee_{rg=f}(u_f\wedge v_g)\right)_{\vartheta_1f=a} \\
	&= \textstyle\left(u_f\wedge \left(\bigvee_{rg=f}v_g\right)\right)_{\vartheta_1f=a} \\ 
	&=\textstyle(u_f)_{\vartheta_1f=a}\wedge \left(\bigvee_{rg=f}v_g\right)_{\vartheta_1g=b} \\
	&=\textstyle(u_f)_{\vartheta_1f=a}\wedge \Sigma_r^{\mf{z}L}(v_g)_{\vartheta_1g=b}.
\end{align*}

Finally, given any $\varphi\colon L\to M\in[\ca{C}^{op},\Frm]$, it is easy to see that each $\mf{z}\varphi^a$ preserves finite infima, simply because each $\varphi^a$ does and they are given componentwise in $\mf{z}L(a)$ and $\mf{z}M(a)$. Hence, $\mf{z}\varphi$ is a morphism in $\Frm[\ca{C}^{op},\Set]$, and so we have shown that $\mf{z}$ restricts also to a functor $[\ca{C}^{op},\Frm]\to\Frm[\ca{C}^{op},\Set]$. We can then state the main result of this section.
\begin{theorem}\label{Thm: z left adjoint for internal frames}
	The functor $\mf{z}\colon[\ca{C}^{op},\Frm]\to\Frm[\ca{C}^{op},\Set]$ is left adjoint to the forgetful $U\colon \Frm[\ca{C}^{op},\Set]\to[\ca{C}^{op},\Frm]$.
\end{theorem}
\begin{proof}
	It suffices to show that the unit and co-unit of the adjunction at the level of suplattices restrict to the corresponding subcategories.
	
	If $L\colon\ca{C}^{op}\to\Frm$ is a presheaf of frames, then it is easy to see that each $\eta_L^a\colon La\to\mf{z}L(a)$ preserves finite infima, by virtue of them being given componentwise in $\mf{z}L(a)$ and the fact that all $Lf$ preserve them.
	
	For the co-unit $\varepsilon_L \colon \mathfrak{z}L \to L$, for any internal frame $L\in\Frm[\ca{C}^{op},\Set]$, we need to show that each $\varepsilon_L^a$ now also preserves finite infima.  We have to show that for every $(u_f)_{\vartheta_1f=a}$ and $(v_f)_{\vartheta_1f=a}$ in $\mf{z}L(a)$
	$$\textstyle\bigvee _{\theta_1 f = a} \Sigma _f u_f \wedge \bigvee _{\theta_1 f = a}  \Sigma _f v_f \leq \bigvee _{\theta_1 f = a} \Sigma _f (u_f \wedge v_f ),$$
	since the other inequality clearly holds. But
	$$\textstyle\bigvee _{\theta_1 f = a}  \Sigma _f u_f \wedge \bigvee _{\theta_1 f = a}  \Sigma _f v_f = \bigvee _{\theta_1 f = a}  \bigvee _{\theta_1 g = a}  (\Sigma _f u_f \wedge \Sigma _g v_g )$$
	for arrows $f,$ $g$ with common codomain $a\in \mathcal{C}.$ We show that each $\Sigma _f u_f \wedge \Sigma _g v_g$
	is dominated by $\Sigma _{fk} (u_{fk} \wedge v_{fk} )$, where $k$ is the map in the pullback
	\begin{displaymath}
		\xy \xymatrix{ \cdot \ar[d]_{k } \ar[r]^{h} & b \ar[d]^{g}\\
			\cdot \ar[r]^{ f}	& a.	}
		\endxy
	\end{displaymath}
	Indeed, we readily calculate that
	\begin{eqnarray*}
		\Sigma _{fk} (u_{fk} \wedge v_{fk} ) &\geq&
		\Sigma _{fk} (u_{fk} \wedge Lk v_{f} )   \\
		&=&\Sigma _{f} \Sigma _k (u_{fk} \wedge Lk v_{f} ) \\
		&=& \Sigma _{f}(\Sigma _k u_{fk} \wedge  v_{f}) \quad \mbox{(Frobenious reciprocity)}  \\
		&=& \Sigma _{f} (  \Sigma _k  u_{gh} \wedge  v_{f})\\
		&\geq& \Sigma _f (\Sigma _k Lh u_g \wedge  v_{f})\\
		&=&\Sigma _f (Lf \Sigma _g  u_g \wedge  v_{f}) \quad \mbox{(Beck-Chevalley)} \\
		&=& \Sigma _g  u_g \wedge \Sigma _f v_{f},
	\end{eqnarray*} 
	completing the proof.
\end{proof}
\begin{remark}
	\cref{Thm: z left adjoint for internal frames} recovers, as a special case, the free internal `difference' frame on a frame equipped with an endomorphism as constructed in \cite[Example 13.11]{tomasic}, i.e.\ when the category $\ca{C}$ is taken to be the monoid $(\mathbb{N},+)$ of natural numbers under addition.
\end{remark}
\begin{corollary}\label{coro:internal_ideals}
	The internal object of ideals $\idl_{\hat{\ca{C}}}L$ of an internal join-semilattice  $L\in [\ca{C}^{op},\nc{Set}]$ 
	is given by $\mathfrak{z} (\idl\circ L) \in \nc{Frm}[\ca{C}^{op},\nc{Set}].$ 
\end{corollary}
\begin{proof}
	Similar \cref{coro:internal-power-object} and \cref{coro:internal-down-closed-object}, taking into account the equivalence
	$$\jslat[\ca{C}^{op},\nc{Set}] \simeq [\ca{C}^{op},\jslat]$$
	and the adjunction 
	$$\idl \dashv u \colon \nc{Frm} \to \jslat$$
	between frames and join-semilattices. 
\end{proof}
\begin{remark}\label{rem:comparision_excomp}
	We take a moment to comment on the similarity of the construction $\mf{z}$ and the \emph{existential completion} $(-)^\exists$ on a presheaf of frames described in \cite{wr_ex_comp} (a frame-theoretic variant of the existential completion for meet-semilattices found in \cite{trotta}).  Suppose, for the purposes of this remark, that $\ca{C}$ is a small category that has all finite limits.  Given a presheaf of frames $L$, the existential completion $(-)^\exists$ freely adds left adjoints $\Sigma_{\pi_b}$ (satisfying Frobenius and Beck-Chevalley, wherever relevant) to the maps $L\pi_b$, for all product projections $\pi_b \colon a \times b\to a$, and so it is no surprise that $L^\exists$ ought to be related to $\mf{z}$, which freely adds left adjoints to all maps.  Indeed, we can explicitly describe $L^\exists$ as a subpresheaf of $\mf{z} L$.  
	
	First, we recall the construction of $L^\exists$ given in \cite[\S 2.9]{wr_ex_comp}.  For an object $a \in \ca{C}$, $L^\exists a$ is the posetal reflection of the following preorder: elements are sets of pairs $\{(b_i,u_i) \mid i \in I\}$ where $b_i \in \ca{C}$ and $u_i \in L(a \times b_i)$ for all $i \in I$, while there is an inequality $\{(b_i,u_i) \mid i \in I\} \leq \{(c_j, v_j) \mid j \in J\}$ if, for each $i \in I$, there is a set of morphisms $\{r_k \colon a \times b_i  \to a \times c_{j_k}  \mid k \in K\}$ such that, for all $k \in K$, $j_k \in J$, the triangle
	\[
	\begin{tikzcd}
	& a \times c_{j_k} \ar{d}{\pi_{c_{j_k}}} \\
	a \times b_i \ar{ru}{r_k} \ar{r}{\pi_{b_i}} & a
\end{tikzcd}
	\]
	commutes, and $u_i \leq \bigvee_{k \in K} L(r_k)(v_{j_k})$.  That is to say, an element of $L^\exists a$ is an equivalence class of the set $\{(b_i,u_i) \mid i \in I\}$ for the equivalence relation $\leq \cap \geq$.
	
	We will demonstrate that $L^\exists$ is isomorphic to the subpresheaf of $\mf{z} L$ given by those those tuples $(u_f)_{\theta_1 f = a} \in \mf{z}L(a)$ satisfying $L(G_f)(u_{\pi_b}) = u_f$ for all $f$ with $\theta_1 f = a$, where here $G_f$ denotes the \emph{graph} of $f$, i.e.\ the unique morphism $G_f \colon b \to a \times b$ making the diagram below commute,
	\[
	\begin{tikzcd}
		& b \ar{ld}[']{f} \ar[equal]{rd} \ar[dashed]{d}{G_f} & \\
		a & \ar{l}[']{\pi_b} a \times b \ar{r}{\pi_a} & b.
	\end{tikzcd}
	\]
	The isomorphism is witnessed, on the one hand, by sending such a tuple to (the equivalence class of) the set $\{(c,u_{\pi_c}) \mid c \in \ca{C}\}$, while in the other direction we send a set $\{(b_i,v_i) \mid i \in I\}$ to the tuple $(\bigvee_{\pi_{b_i} g = f} L(g)(v_i))_{\theta_1 f = a}$.  To see why this defines an isomorphism, note firstly that $u_f = L(G_f)(u_{\pi_b}) \leq \bigvee_{\pi_c g = f } L(g)(u_{\pi_c})$, but also $\bigvee_{\pi_c g = f } L(g)(u_{\pi_c}) \leq u_f$ by lax-compatibility of the family $(u_f)_{\theta_1 f =a} \in \mf{z}L(a)$, and so we obtain one of the desired equalities $(u_f)_{\theta_1 = a} = (\bigvee_{\pi_{c} g = f} L(g)(u_{\pi_c}))_{\theta_1 f = a}$.  It remains to argue that the sets $\{(b_i,v_i) \mid i \in I\}$ and $\{(c,\bigvee_{\pi_{b_i} g = \pi_{c}} L(g)(v_i)) \mid c \in \ca{C}\}$ are equivalent.  It is easily shown that one inequality, $\{(b_i,v_i) \mid i \in I\} \leq \{(c,\bigvee_{\pi_{b_i} g = \pi_{c}} L(g)(v_i)) \mid c \in \ca{C}\}$, is witnessed by the identity on each $b_i$, while for the converse inequality it suffices to consider the set of arrows $\{g \colon a \times c \to a \times b_i \mid \pi_{b_i} g = \pi_c\}$.
\end{remark}
\begin{remark}[cf.\ Remark 2.8 \cite{wr_ex_comp}]
	We gave above an explicit construction of $\mf{z}$, and a direct proof that it defines the left adjoint to the forgetful functor $U \colon \Frm[\ca{C}^{op},\Set] \to [\ca{C}^{op},\Frm]$.  But how might we have arrived at this description organically?  We explain here how to obtain the construction $\mf{z}$ from the \emph{geometric completion on a doctrinal site} given in \cite{car,wr,wrphd}.
	
	We first observe that, given a presheaf of frames $L$, the data of the suprema in each section of $L$ can be captured by a certain Grothendieck topology.  Consider the \emph{Grothendieck construction} $\ca{C} \rtimes L$, i.e.\ the category whose objects are pairs $(a,u)$ of an object $a \in \ca{C}$ and an element $u \in La$, and whose arrows $f \colon (b,v) \to (a,u)$ are those arrows $f \colon b \to a$ such that $v \leq Lf(u)$.  The category $\ca{C} \rtimes L$ comes equipped with a canonical fibration $p \colon \ca{C}\rtimes L \to \ca{C}$, given by $(a,u) \mapsto a$, and since each section has a top element (and these are preserved by transition maps), there is a functor $t \colon \ca{C} \to \ca{C} \rtimes L$, given by $a \mapsto (a,1)$, which is right adjoint $p$.  Obverse that the category $\ca{C} \rtimes L$ can be equipped with a Grothendieck topology $J^L$ whose covering sieves are generated by families of the form
	\[\textstyle \left\{\id_a \colon (a,u_i) \to \left(a,\bigvee_{i \in I} u_i\right) \mid i \in I\right\}.\]
	(This is the same Grothendieck topology considered in \cite[Remark 2.12]{wr_ex_comp}.) Let $M$ be another presheaf of frames, and let $\phi \colon U \circ L \Rightarrow U \circ M$ be a natural transformation where each component $\phi^a$ is a meet-semilattice homomorphism, i.e.\ $\phi$ describes a natural transformation between {meet-semilattice} valued functors. Note that each $\phi^a$ also preserves suprema, and hence $\phi$ is a natural transformation of presheaves of frames, if and only if the corresponding morphism of fibrations $\tilde{\phi} \colon \ca{C} \rtimes L \to \ca{C} \rtimes M$, given by $(a,u) \mapsto (a,\phi^a(u))$, sends $J^L$-covering families to $J^M$-covering families. In other words, we are able to capture the data of suprema in $L$ and $M$ via Grothendieck topologies.
	
	It follows from \cite[Theorem IV.25]{wrphd} that, not only can $L$ be freely completed to an internal frame, but that free completion can be explicitly calculated by taking the subobject classifier of the corresponding sheaf topos $\Sh(\ca{C}\rtimes L,J^L)$ and precomposing with the opposite of $t$, i.e.\ the free completion is given by
	\[
	\Omega_{\Sh(\ca{C}\rtimes L,J^L)} \circ t^{op} \colon \ca{C}^{op} \to \Frm.
	\]
	Indeed, it is not too hard to massage the description of $\Omega_{\Sh(\ca{C}\rtimes L,J^L)} \circ t^{op}$, in terms of a $J^L$-closed sieves, in order to witness a natural isomorphism $\mf{z} L \cong \Omega_{\Sh(\ca{C}\rtimes L,J^L)} \circ t^{op}$ --  an element $R \in \Omega_{\Sh(\ca{C}\rtimes L,J^L)} \circ t(a)$, i.e.\ a $J^L$-closed sieve on the object $(a,1)$, is sent to the tuple 
	\[(\bigvee \{v \in Lb \mid f \colon (b,v) \to (a,1) \in R\})_{\theta_1 f = a} \in \mf{z}L(a),\]
	while in the other direction an element $(u_f)_{\theta_1 f = a} \in \mf{z}L(a)$ is sent to the sieve $\{f \colon (b,v) \to (a,1) \mid v \leq u_f\}$.  We omit the details that this does in fact yield a natural isomorphism.
\end{remark}
\paragraph{Presheaves on a groupoid.}
For an $L\in\Frm[\ca{C}^{op},\Set]$, the unit $\eta_L\colon L\to\mf{z}L$ need not be a frame morphism since, as we commented on earlier, it need not be a suplattice homomorphism. The latter requirement says that for all $r \colon a \to b\in\ca{C}$ the square
\begin{center}
	\begin{tikzcd}
		La\ar[d,"\Sigma_r^L"']\ar[r,"\eta_a"] & \mf{z}La\ar[d,"\Sigma_{r}^{\mf{z}L}"] \\
		Lb\ar[r,"\eta_b"] & \mf{z}Lb
	\end{tikzcd}
\end{center}
commutes. Explicitly, this means that, for all $u\in La,$ 
\begin{equation}\label{eq:grpd_cond}
\textstyle	Lf \Sigma _r u  = \bigvee _{ rg=f} Lgu 
\end{equation}
This being the case for all $L$ forces the topos $[\ca{C}^{op},\Set]$ to be Boolean.

\begin{proposition}\label{Groupoid characterization}
	For a small category $\ca{C}$ (with weak pullbacks), the following are equivalent:
	\begin{enumerate}[label = {(\arabic*)}]
		\item\label{enum:C_grpd} $\ca{C}$ is a groupoid;
		\item\label{enum:unit_morphism} For every internal frame $L\in\Frm[\ca{C}^{op},\nc{Set}]$, the unit $\eta_L\colon L\to\mf{z}L$ is a morphism of frames;
		\item\label{enum:unit_iso} For every presheaf of frames $L\in[\ca{C}^{op},\Frm]$, the unit $\eta_L\colon L\to\mf{z}L$ is an isomorphism;
		\item\label{enum:U_ff} The forgetful functor $U \colon \Frm[\ca{C}^{op},\Set] \to [\ca{C}^{op},\Frm]$ is fully faithful;
		\item\label{enum:U_eqv} The forgetful functor $U \colon \Frm[\ca{C}^{op},\Set] \to [\ca{C}^{op},\Frm]$ is an equivalence of categories.
	\end{enumerate}
\end{proposition}
\begin{proof}
	Note that the implications
	\[
	\text{\cref{enum:U_eqv}} \implies \text{\cref{enum:U_ff}} \iff \text{\cref{enum:unit_iso}} \implies \text{\cref{enum:unit_morphism}}
	\]
	are all trivial consequences, either by definition or by well-known categorical arguments.  For the implication \cref{enum:C_grpd}$\implies$\cref{enum:U_eqv}, if $\ca{C}$ is a groupoid, then every presheaf of frames $L \colon \ca{C}^{op} \to \Frm$ has left adjoints to the transition maps $Lf$, since these are just given by $Lf^{-1}$.  Moreover, if
	\[
	\begin{tikzcd}
		a \ar{r}{f} \ar{d}[']{g} & b \ar{d}{h} \\
		c \ar{r}{k} & d
	\end{tikzcd}
	\]
	is a commuting square, then by functoriality of $L$ we have that $Lh Lf = Lk Lg$, and so the Beck-Chevalley condition $\Sigma_k Lh = (Lk)^{-1} Lh = Lg (Lf)^{-1} = Lg \Sigma_f$ is satisfied.  Checking Frobenius reciprocity is similarly a consequence of the fact that $\Sigma_f$ is given by $Lf^{-1}$.  Thus, every presheaf of frames is an internal frame, and it is just as easy to check that every natural transformation yields a morphism of internal frames.
	
	It remains to prove the implication \cref{enum:unit_morphism}$\implies$\cref{enum:C_grpd}. Take $L$ to be the constant presheaf $L\colon\ca{C}^{op}\to\Frm$ with value $\mathbf{2}=\{0\leq 1\}$. It is easy to see that this is an internal frame in $[\ca{C}^{op},\Set]$. Since $\eta_L$ is an internal frame homomorphism, for any two morphisms $f,r$ in $\ca{C}$ with $\vartheta_1f=\vartheta_1r$ the condition \cref{eq:grpd_cond} above must hold. Taking $u=1$ we have that the left hand side is $1$, since $Lf$ and $\Sigma_r$ are both the identity maps. Thus, the supremum on the right hand side must be nonempty, showing that $f$ must factor through $r$. In other words, any two morphisms with common codomain must factor through each other, which implies that $\ca{C}$ is a groupoid.
\end{proof}

\section{Applications}
With the construction $\mf{z}$ in hand, we now turn to a study of the interaction of properties of an internal locale and properties on their sections.  
\subsection{Local compactness}
An internal locale $L$ is (internally) locally compact if there exists a map $\Lambda \colon L \to \idl_{\hat{\ca{C}}}L$ which is left adjoint to the \emph{supremum map}, which we denote $S \colon\idl_{\hat{\ca{C}}}{L} \to L$.  We show that an internal locale that is locally compact has locally compact sections, and moreover transition maps preserve the way-below relation, but we also show that the internal way-below relation is not given pointwise.

Before continuing, note that an adjunction of maps (i.e.\ natural transformations) between internal posets can be entirely checked by passing to sections, i.e.\ an adjunction $\phi \dashv \psi \colon X \to Y \in \Pos[\ca{C}^{op},\Set]$ gives, for all $a \in \ca{C}$, an adjunction $\phi_a \dashv \psi_a \colon Xa \to Ya$, and vice versa.

First, we explicitly describe the internal supremum map $S \colon \idl_{\hat{\ca{C}}} L \to L$, using the fact that $\idl_{\hat{\ca{C}}}L = \mf{z}(\idl \circ L)$.  The map $S$ is left adjoint to the canonical inclusion $\zeta \colon L \hookrightarrow \idl_{\hat{\ca{C}}}L$, which is given by $\zeta_a(u) = (\downarrow\!Lf(u))_{\theta_1 f = a} \in \mf{z}(\idl \circ L)(a)$.  Hence, the component of $S$ at $a \in \ca{C}$ must be given by $S_a\left((I_f)_{\theta_1 f = a}\right) = \bigvee \{\bigcup_{\theta_1 f = a} \Sigma_f[I_f]\}$ (where we recall our convention that $\Sigma_f[I_f]$ denotes the ideal generated by the direct image of $I_f$ under $\Sigma_f$) since, given $u \in La$, by a direct calculation we have that:
\begin{align*}
	S_a\left((I_f)_{\theta_1 f = a}\right) \leq u & \Leftrightarrow \forall f \, \forall v \in I_f \, (\Sigma_f v \leq u), \\
	& \Leftrightarrow \forall f \, \forall v \in I_f \, (v \leq Lf(u)), \\
	& \Leftrightarrow \forall f \, (I_f \subseteq \, \downarrow\! Lf(u)), \\
	& \Leftrightarrow (I_f)_{\theta_1 f = a} \leq \zeta_a(u).
\end{align*}

\begin{proposition}\label{Internal LC implies sectionwise LC}
If $L \in \Frm[\ca{C}^{op}, \Set]$ is locally compact, then all its sections are locally compact.
\end{proposition}
\begin{proof}
We define a left adjoint $\lambda _a \colon La \to \idl (La)$ to the supremum map $\bigvee \colon \idl(La) \to La$ (of the section)
by  $\lambda _a (u) = \langle \bigcup_{\theta_1 f = a} \Sigma_f [(\Lambda _a u)_f]  \rangle$.
Indeed, we have that
\begin{eqnarray*}
\textstyle \lambda _a (u) = \langle \bigcup _{\theta_1 f = a} \Sigma _f [(\Lambda _a u)_f] \rangle \subseteq I &\Leftrightarrow &
\forall f  \; (\Sigma _f [(\Lambda _a u)_f] \subseteq I)  \\
&\Leftrightarrow & \forall f  \; ((\Lambda _a u)_f \subseteq  Lf[I])\\
&\Leftrightarrow &  \Lambda _a (u) \leq (  Lf[I])_{\theta_1 f = a} \\
&\Leftrightarrow & u \leq S_a (( Lf[I] )_{\theta_1 f = a}).
\end{eqnarray*} 
Finally, we note that $S_a (( Lf[I] )_{\theta_1 f = a}) = \bigvee \{\bigcup_{\theta_1 f = a} \Sigma_fLf[I]\} = \bigvee I$, since $I \subseteq \bigcup_{\theta_1 f = a} \Sigma_fLf[I]$ and $\bigcup_{\theta_1 f = a} \Sigma_fLf[I] \subseteq I$, where the latter follows from the fact that $\Sigma_f$ is left adjoint to $L_f$.  Hence, $\lambda_a$ does define a left adjoint.  In other words, we have shown that $\langle \bigcup _{\theta_1 f = a} \Sigma _f [(\Lambda _a u)_f] \rangle$ defines the way-below segment of the element $u \in La.$
\end{proof}

\begin{proposition}\label{Internal LC implies preservation of way-below}
If $L \in \Frm[\ca{C}^{op}, \Set]$ is locally compact, then all the transitions $Lg \colon La \to Lb,$ 
for $g\colon b\to a,$ preserve the way-below relation of the sections, i.e if $w\ll u \in La,$ then $Lg(w) \ll Lg(u)$ in $Lb.$
\end{proposition}
\begin{proof}
In view of the definition of the way-below segment of an element $u \in La$ from the proof of \cref{Internal LC implies sectionwise LC}, we have to show that 
if 
$$\textstyle w\in \lambda _a (u) =  \langle \bigcup _{\theta_1 f = a} \Sigma _f [(\Lambda _a u)_f] \rangle,$$ 
then 
$$\textstyle Lg(w) \in \lambda _b (Lgu) = \langle \bigcup _{\theta_1 h = b} \Sigma _h [(\Lambda _b Lgu)_h] \rangle ,$$
or in other words that 
$$\textstyle Lg[\langle \bigcup _{\theta_1 f = a} \Sigma _f [(\Lambda _a u)_f] \rangle] \subseteq \langle \bigcup _{\theta_1 h = b} \Sigma _h [(\Lambda _b Lgu)_h] \rangle.$$
Since direct images preserve unions, if suffices to show that, for any $f \colon c \to a$, $Lg[\langle \Sigma _f [(\Lambda _a u)_f] \rangle] \subseteq \Sigma _h [(\Lambda _a Lg u)_{h}] \subseteq \lambda_b(Lgu)$, where $h$ is the pullback of $f$ along $g$, as in the diagram
\[
\begin{tikzcd}
	c \times_a b \ar{r}{h} \ar{d}[']{k} & b \ar{d}{g} \\
	c \ar{r}{f} & a.
\end{tikzcd}
\]
Thus, using consecutively the Beck-Chevalley condition for $L$, lax compatibility of the family $(\Lambda _a u)_{\theta_1 f = a}$, and the naturality of $\Lambda$, we obtain the inclusion
\begin{eqnarray*}
Lg[\langle \Sigma _f [(\Lambda _a u)_f] \rangle] &=&  \Sigma _h Lk [(\Lambda _a u)_f ] \\
                  &\subseteq&  \Sigma _h [(\Lambda _a u)_{fk}]  \\
                  &=&  \Sigma _h [(\Lambda _a u)_{gh}]\\ 
     			  &=& 
     \Sigma _h [(\Lambda _b Lg u)_{h}], 
\end{eqnarray*} 
as desired.
\end{proof}

\begin{remark}\label{Remark CCD 1} In the presence of local compactness of the sections, the condition that the transitions 
$Lf \colon La \to Lb $ preserve the way-below relation is equivalent to the condition that their right adjoints 
$Lf _*$ preserve directed suprema. This comes in analogy to the results of \cite{ccd}, where the authors show that an internal constructively completely distributive (ccd) lattice (i.e an $L\in \nc{SupLat}[\ca{C}^{op}, \Set]$ with the supremum map $S \colon L \to L $ admitting a left adjoint) is an internal suplattice such that the sections $La$ are ccd lattices and the right adjoints to the transitions $Lf \colon La \to Lb$ preserve suprema. 
\end{remark}

\begin{proposition}\label{Prop: Interal way-below preserving maps}
	Let $L$ and $M$ be internal locally compact frames.  If $\varphi\colon L\to M\in\Frm[\ca{C}^{op},\Set]$ is a morphism which preserves the internal way-below relation, then for each object $a\in\ca{C}$ the frame morphism $\varphi^a\colon La\to Ma$ preserves the way-below relation of the sections.
\end{proposition}
\begin{proof}
Given the internal left adjoint $\Lambda^L\colon L\to \idl_{\hat{\ca{C}}} L$ to the supremum morphism $S^L\colon \idl_{\hat{\ca{C}}} L\to L$, recall from \cref{Internal LC implies sectionwise LC} that the way-below segment $\twoheaddownarrow \! u$ of any $u\in La$ is given by $\lambda^L _a (u) = \langle \bigcup_{\theta_1 f = a} \Sigma^L_f [(\Lambda^L _a u)_f]  \rangle$. Hence, we must show that $\varphi^a[\lambda^L_a(u)]\subseteq\lambda^M_a(\varphi^a(u))$ as subsets of $Ma$.

The fact that $\varphi$ preserves the way-below relation internally means that we have a lax commutative square
\begin{center}
\begin{tikzcd}
	L\ar[r,"\varphi"]\ar[d,"\Lambda^L"']\ar[dr,phantom,"\leq"] & M\ar[d,"\Lambda^M"] \\
	\idl_{\hat{\ca{C}}} L\ar[r,"\idl_{\hat{\ca{C}}}\varphi"'] & \idl_{\hat{\ca{C}}}M
\end{tikzcd}
\end{center}
where $\idl_{\hat{\ca{C}}}\varphi$ is the morphism sending any $(I_f)_{\vartheta_1 f=a} \in \idl_{\hat{\ca{C}}} L(a)$ to the element $(\downarrow\varphi^{\vartheta_0f}[I_f])_{\vartheta1f=a} \in \idl_{\hat{\ca{C}}} M(a)$. Hence, we have an inclusion $\downarrow\varphi^{\vartheta_0f}[(\Lambda^L_au)_f]\subseteq\Lambda^M_a(\varphi^a(u))_f$ for every $u\in La$ and $f$ with $\vartheta_1f=a$. Using this inclusion, alongside the fact that $\varphi$ preserves internal suprema, we have that
\begin{align*}
	\varphi^a[\lambda^L_a(u)] &= \varphi^a\left[\left\langle{\textstyle \bigcup_{\theta_1 f = a} \Sigma^L_f [(\Lambda^L _a u)_f] } \right\rangle\right] \\
							  &\subseteq \left\langle \varphi^a\left[{\textstyle\bigcup_{\theta_1 f =a} \Sigma^L _f (\Lambda^L _a u)_f} \right] \right\rangle \\
						 	  &= \left\langle {\textstyle \bigcup_{\theta_1 f = a} } \varphi^a\Sigma^L _f [(\Lambda^L _a u)_f] \right\rangle \\
						 	  &= \left\langle  {\textstyle \bigcup_{\theta_1 f = a} } \Sigma^M_f\varphi^{\vartheta_0f} [(\Lambda^L _a u)_f] \right\rangle \\
						 	  &\subseteq \left\langle  {\textstyle \bigcup_{\theta_1 f = a} } \Sigma^M_f[\Lambda^M_a(\varphi^a(u))_f] \right\rangle = \lambda^M_a(\varphi^a(u)),
\end{align*}
completing the proof.
\end{proof}

Let us denote by $\nc{LCFrm}_{\ll}$ the category whose objects are locally compact frames and whose morphisms are frame homomorphisms which preserve the way-below relation. As usual, we denote by $\nc{LCFrm}_{\ll}[\ca{C}^{op},\Set]$ the corresponding category of locally compact frames and way-below preserving frame maps internal to $[\ca{C}^{op},\Set]$. The previous three propositions taken together thus show that there is a forgetful functor $U\colon \nc{LCFrm}_{\ll}[\ca{C}^{op},\Set]\to [\ca{C}^{op},\nc{LCFrm}_{\ll}]$. In fact, as is clear from the proofs, this works already at the level of continuous lattices and way-below preserving suplattice homomorphisms, i.e. there is a forgetful functor $U\colon \nc{ContLat}_{\ll}[\ca{C}^{op},\Set]\to [\ca{C}^{op},\nc{ContLat}_{\ll}]$ (cf.\ \cite{towncah}).

\begin{example}[The internal way-below relation]
We have seen that local compactness of an internal frames implies local compactness of its sections.  In a similar fashion, the internal way-below relation $\ll_L$ of an internal locally compact frame $L$ is contained in the way-below relation of each of its sections.  More precisely, the internal way-below relation defines a subobject ${\ll_L} \subseteq L \times L$, and so for each object $a \in \ca{C}$ we obtain a relation ${\ll_L}(a) \subseteq La \times La$ that is contained in the way-below relation ${\ll _{La} } \subseteq La \times La$ of the (locally compact) section $La$. Indeed $ (u,v) \in {\ll _L}(a) $ entails that for all $(I_f)_{\vartheta_1 f=a} \in \idl_{\hat{\ca{C}}}L (a)$, if 
$v\leq S_a( (I_f)_{\theta_1 f = a}) = \bigvee \{\bigvee_{\theta_1 f = a} \Sigma_f [I_f]\}$ then $(\downarrow \! Lf(u))_{\theta_1 f = a} \leq (I_f)_{\theta_1 f = a} $, which is equivalent to requiring that $u\in I_{\id_a}$.  Now if $v \leq \bigvee I $, for an ideal $I \in \idl(La)$, then
$\bigvee I \leq \bigvee \{\bigvee_{\theta_1 f = a} \Sigma_f [Lf[I]]\}$, and hence $u \in I$.  Thus, $(u,v) \in {\ll_{La}}$.

However, the inclusion ${\ll_L}(a) \subseteq {\ll_{La}} $ may be strict, as we now show in an example.  Thus, the internal way-below relation is, in general, markedly weaker than the way-below relation of each section.

Consider the presheaf topos on $\mathbf{2} = \{0\leq 1\}$ and the subobject classifier $\Omega \in [\mathbf{2}^{op}, \Set]$, which is a compact, and hence locally compact, internal frame. (Indeed, the subobject classifier is a compact frame in any topos, see \cite{t}, in essence because $\Omega$ is the internalisation of the one-point space; recall also that the topos of presheaves on $\mathbf{2} = \{0 \leq 1\}$ is equivalent to the topos of sheaves on the Sierpi\'{n}ski space.)  A simple calculation (using \cite[\S I.4]{sgl} or \cref{ex:comuting_subobj_class} if necessary) shows that the sections of the subobject classifier $\Omega$ are given by
\[\Omega _0 = \{0 \leq 1\}, \qquad \Omega _1 = \{0 \leq \# \leq 1\}.\]
Since both sections are finite chains, the way-below relation on each of them coincides with the order relation, and in particular
\[{\ll _{\Omega_1 }} = \{(0,0 ), (0,\#) ,(0,1), (\#,\#), (\#,1), (1,1) \}.\]
We will show that $(\#,\#) \not \in {\ll_\Omega}(1)$.

Note that the ideals on the sections $\Omega_0$, $\Omega_1$ are, respectively,
\[
\idl(\Omega_0) = \{\{0\},\{0,1\}\}, \qquad \idl(\Omega_1) = \{\{0\},\{0,\#\},\{0,\#,1\}\},
\]
and recall that an element of $\idl_{\hat{\mathbf{2}}} \Omega(1) = \mf{z}(\idl \circ \Omega)(1)$, i.e.\ an internal ideal, consists of a pair $(I_\leq, I_{\id_1})$ where $I_{\leq} \in \idl(\Omega_0)$, $I_{\id_1} \in \idl (\Omega_1)$, satisfying the lax-compatibility condition $\Omega(\leq)[I_{\id_1}] \subseteq I_{\leq}$.  To demonstrate that $(\#,\#) \not \in {\ll_\Omega}(1)$, it suffices to find such a pair $(I_\leq, I_{\id_1})$ such that $\# \leq S_1((I_\leq, I_{\id_1}))$ but for which $(\downarrow\!\Omega(\leq)(\#),\downarrow\!\#) = (\{0,1\},\{0,\#\}) \not \subseteq (I_\leq, I_{\id_1})$.  Taking $I_{\leq} $ to be $\{0,1\}$ and $I_{\id_1}$ to be $\{0\}$ satisfies these conditions, since 
\[S_1((I_{\leq},I_{\id_1})) = \Sigma_{\leq} 0 \lor \Sigma_{\leq} \# \lor \Sigma_{\id_1} 0 = 0 \lor \# \lor 0 = \# \]
but clearly $\downarrow\! \# = \{0,\#\} \not \subseteq \{0\}$.  Thus, there is a strict inclusion ${\ll _{\Omega }}(1) \subsetneqq {\ll _{\Omega _1 }}$.
\end{example}
\subsection{Stable local compactness}
A locally compact locale is said to be \emph{stably locally compact} when the way-below relation
is stable under binary infima. When combined with the notion of compactness, which we study below in \cref{subsec:compactness}, we obtain the notion of a \emph{stably compact} frame. We will show that the (internal) stable local compactness of an internal frame is inherited by all its sections.  To do so, we shall use the following general lattice theoretic result:
\begin{lemma}\label{Lemma: ideal generated by binary intersection}
	For any subsets $S, \; T \subseteq L$ of a distributive lattice, for the corresponding ideals they generate in $L$ we have that
	$$\langle S \cap T \rangle = \langle S \rangle \cap \langle T \rangle.$$
\end{lemma}
\begin{proof}
	We only need to show the inclusion of $\langle S \rangle \cap \langle T \rangle$ into $\langle S \cap T \rangle$.  Let $x \in \langle S \rangle \cap \langle T \rangle$ be such that 
	$x \leq s_1 \vee \cdots \vee s_m ,$ with $s_i \in S $ and $x\leq t_1 \vee \cdots \vee t_n,$ with $t_j \in T.$
	Then
	$$\textstyle x\leq \bigvee _i s_i \wedge \bigvee _j t_j = \bigvee _{i,j} (s_i \wedge t_j ),$$
	with each $s_i \wedge t_j \in S\cap T$, and hence $x \in \langle S \cap T \rangle .$
\end{proof}

\begin{lemma}\label{Lemma: Frobenius for ideals}
	For $L \in \Frm[\ca{C}^{op}, \Set],$
	if  $I \subseteq La,$ $J \subseteq Lb$ are ideals and $f\colon a \to b \in \mathcal{C},$ then
	$$\Sigma _f [I \cap Lf[J]] = \Sigma _f [I] \cap J ,$$
	i.e.\ we have Frobenius reciprocity for section-wise ideals.
\end{lemma}
\begin{proof}
	For the non-trivial inclusion, let
	$v \in \Sigma _f [I] \cap J, $ so that $v\leq \Sigma _f u $ with $u\in I $ and $v\in J.$
	Then $$v = v \wedge \Sigma _f u = \Sigma _f (Lf v \wedge u ),$$
	where $Lf v \wedge u \leq Lfv \in Lf[J],$ while $Lf v \wedge u \leq u \in I.$
	This means that $Lf v\wedge u \in Lf[J] \cap I,$ and so 
	$v= \Sigma _f (Lf v \wedge u ) \in \Sigma _f [I \cap Lf[J]]$, completing the proof.
\end{proof}
\begin{proposition}\label{Prop: sections of stably locally compact}
	If $L \in \Frm[\mathcal{C}^{op}, \Set]$ is stably locally compact, then all its sections $La$ are stably locally compact.
\end{proposition}
\begin{proof}
	The internal frame $L$ being stably locally compact means that $\Lambda \colon L \to \idl_{\hat{\ca{C}}}L$ preserves binary infima, and hence $\Lambda _a \colon La \to \idl_{\hat{\ca{C}}}L(a)$ preserves binary infima for all objects $a \in \mathcal{C}$. Using the description of the pointwise way-below relation from \cref{Internal LC implies sectionwise LC} and \cref{Lemma: ideal generated by binary intersection}, we then have that
	\begin{eqnarray*}
		\lambda _a (u) \cap \lambda _a (w) &=& 
		\textstyle \langle \bigcup_{\theta_1 f = a} \Sigma _f (\Lambda _a u)_f \rangle \cap \langle \bigcup_{\theta_1 f = a} \Sigma _f (\Lambda _a w)_f \rangle \\
		 &=& \left\langle \left({\textstyle\bigcup _{\theta_1 f = a} \Sigma _f (\Lambda _a u)_f ) \cap  \bigcup _{\theta_1 f =a} \Sigma _f (\Lambda _a w)_f} \right) \right\rangle
		\\
		&= & \left\langle {\textstyle \bigcup _{\theta_1 f = \theta_1 g = a} } \left( \Sigma _f (\Lambda _a u)_f  \cap  \Sigma _g (\Lambda _a w)_g \right) \right\rangle
	\end{eqnarray*}
	From this point on, the Frobenius reciprocity for pointwise-ideals from \cref{Lemma: Frobenius for ideals} ensures that we can still follow the calculation in the proof of \cref{Thm: z left adjoint for internal frames}, and so we deduce that each $\Sigma _f (\Lambda _a u)_f  \cap  \Sigma _g (\Lambda _a w)_g $ is contained in 
	$$\Sigma _{fk} [(\Lambda _a u)_{fk} \cap (\Lambda _a w)_{fk}] = \Sigma _{fk} [(\Lambda _a (u \wedge w)_{fk}]  \subseteq \lambda _a (u \wedge w),$$
	where $k$ denotes the pullback of $g$ along $f$.  This completes the non-trivial inclusion $\lambda_a(u) \cap \lambda_a(w) \subseteq \lambda_a(u \land w)$, and thus we conclude that $\lambda_a(u) \cap \lambda_a(w) = \lambda_a(u \land w)$ as desired.
\end{proof}

\subsection{Compactness}\label{subsec:compactness}
A frame $L$ is \emph{compact} if whenever the supremum of an ideal is the top element of $L$, then the ideal contains the top element, and thus must be given by $\{L\}$ itself. In other words the, equalizer 
\[
\begin{tikzcd}
	E \ar{r} & \idl L \ar{rr}{\bigvee} \ar{rd} && L \\
	&& \mathbf{1} \ar{ru}[']{1} 
\end{tikzcd}
\]
in the category of partially ordered sets,
of the supremum and the constantly equal to the top element is $\{L\}$. From this characterisation, we derive that the sections of an internally compact frame are compact as well, but we show also that the converse does not hold.
\begin{proposition}\label{Prop: sections of compact}
If an internal frame $L \in \Frm[\mathcal{C}^{op}, \Set]$ is compact, then all its sections $La$ are compact.
\end{proposition}
\begin{proof}
The diagrammatic condition expressing compactness above remains an equalizer when evaluated at any $a\in \mathcal{C}$, since limits and colimits in $[\ca{C}^{op},\Set]$ are computed pointwise.
Thus, whenever a lax compatible family $(I_f)_{\theta_1 f = a} \in \idl_{\hat{\ca{C}}}L(a) = \mf{z}(\idl \circ L)(a)$ is chosen for which the (internal) supremum
\[
S_a( (I_f)_{\theta_1 f = a} )= \bigvee {\textstyle \{ \bigcup _{\theta_1 f = a} \Sigma _f [I_f]\}}
\]
is equal to the top element $1 \in La$, then $(I_f)_{\theta_1 f = a}$ equals the family $(L(\theta_0 f))_{\theta_1 f = a}$.
Now given an ideal $I \subseteq La$ with $\bigvee I=1_{La}$, the family $(Lf[I])_{\theta_1 = a} \in \idl L(a)$ has the property that $S_a((Lf[I])_{\theta_1 = a}) = 1$, because 
$I$ participates in the union (via the $f=\id_a $ component) and so
\[
\bigvee I \leq  \bigvee {\textstyle \{ \bigcup _{\theta_1 f = a} \Sigma _f [Lf[I]]\}} \leq \bigvee I =1.
\] 
Hence, by the internal compactness of $L$, we have $(Lf[I])_{\theta_1 f = a} = (L(\vartheta _0 f))_{\theta_1 f = a}$.  By taking the $\id_a$ component of the tuple, we conclude that $I =La$, demonstrating that $La$ is compact.
\end{proof}
By combining \cref{Prop: sections of compact} and \cref{Prop: sections of stably locally compact}, we get the following corollary:
\begin{corollary}
	If $L \in \Frm[\ca{C}^{op},\Set]$ is stably compact, then so too are all its sections.
\end{corollary}
\begin{example}
	The converse to \cref{Prop: sections of compact} does not hold.  Consider the constant presheaf $L \colon \mathbf{2}^{op} \to \Frm$ on a (non-trivial) compact frame $L$, which defines an internal frame of $[\mathbf{2}^{op},\Set]$ (recall that $\mathbf{2}$ denotes the category $\{0 \leq 1\}$).  We can show that $L$ is not internally compact by demonstrating that the equalizer of sets
	\[
	\begin{tikzcd}
		E(1) \ar[hook]{r} & \mf{z}(\idl \circ L)({1}) \ar[shift left ]{r}{S_\mathbf{1}} \ar[shift right ]{r}[']{1} & L({1})
	\end{tikzcd}
	\]
	contains, in addition to the maximal internal ideal $(L)_{\theta_0 f = {1}}$, a non-maximal tuple $(I_f)_{\theta_0 f = {1}} \in \mf{z}(\idl \circ L)({1})$.
	
	We take the family $(I_{\leq},I_{\id_1})$ where $I_{\id_1} = \{0\} \neq L$ and $I_\leq = L$, which is a lax-compatible family of ideals, and so $(I_{\leq},I_{\id_1}) \in \mf{z}(\idl \circ L)({1})$.  We then calculate that
	\[
		S_{1}((I_{\leq},I_{\id_1})) = \bigvee \left\{ { \Sigma_{\leq}[L] \cup \Sigma_{\id_1}[\{0\}] }\right\} = \bigvee \{L \cup \{0\}\} = 1,
	\]
	and yet $(L)_{\theta_0 f = {1}} \neq (I_f)_{\theta_0 f = {1}}$, proving that $L$ is not internally compact as desired.
\end{example}

\subsection{Hausdorffness}
A locale is Hausdorff if its diagonal $\Delta \colon  L \to L \times L$ is closed.
A map of locales $f \colon L \to M$ is closed if, for all $u\in L,$ $v\in M,$ $f_* (u \vee f^* v) = f_* u \vee v$, which can be expressed diagrammatically as the commutativity of the following  diagram of posets
\[\xy \xymatrix{ L \times M \ar[rr]^{f_* \times \id_M} \ar[d]_{f^*} & & M \times M \ar[dd]^{-\vee - }\\
 L \times L \ar[d]_{- \vee -} & & \\
	L \ar[rr]^{ f_* }	& & M.	}
		\endxy\]
Hence a map $f \colon L \to M \in \Frm[\mathcal{C}^{op},\Set]$ is closed if and only if 
$f_a  \colon La \to Ma$ is closed for all $a \in \mathcal{C}$.  We show that if $L$ is a presheaf of Hausdorff frames, then $\mf{z}L$ is an internally Hausdorff frame.  As a consequence, we derive explicitly the well-known fact that the subobject classifier of a presheaf topos is internally compact and Hausdorff, and thus we are able to give an example, suggested to us by S.~Henry, where the Hausdorffness of an internal frame is \emph{not} inherited by its sections.
\begin{proposition}\label{prop:z_preserves_closedness}
Given a natural transformation  $\varphi \colon  L\to M  \in [\mathcal{C}^{op},\Frm ]$ such, that for all objects $a \in \ca{C}$, the component $\varphi ^a \colon La \to Ma$ is closed, 
then the induced map $\mathfrak{z}\varphi ^a \colon \mathfrak{z}L(a) \to \mathfrak{z}M (a),$ is a closed frame map, and hence $\mf{z}\phi \colon \mf{z} L \to \mf{z}M$ is a closed morphism of internal locales.
\end{proposition}
\begin{proof}
This follows immediately from the fact that the frame structure on both $\mf{z}L(a)$ and $\mf{z}M(a)$, the action of $\mf{z}\phi^a$, and the action of $\mf{z}\phi^a_\ast$ are all computed pointwise (see \cref{sec:free_const}).
\end{proof}

\begin{proposition}\label{Prop: sections Hausdorff implies zL Hausdorff} 
	Let  $L\in [\mathcal{C}^{op},\Frm]$ be a presheaf of frames such that each section $La$ is Hausdorff. Then $\mathfrak{z}L$ is Hausdorff.
\end{proposition}
\begin{proof}
	Colimits in $[\mathcal{C}^{op},\Frm]$ are computed pointwise, and so the coproduct of two presheaves of frames $L$, $M$ is given by the presheaf $a\mapsto La \otimes Ma$, which we denote by $L\otimes M$. 
	The functor $\mathfrak{z}$, being a left adjoint, preserves colimits and so $\mathfrak{z} (L \otimes M) \cong \mathfrak{z}L \sqcup \mathfrak{z}M,$ where $\mf{z}L\sqcup\mf{z} M$ denotes the coproduct of $\mf{z}L, \mf{z}M \in \Frm[\mathcal{C}^{op},\mathbf{Set}]$ as internal frames.
	Since each $La$ is Hausdorff, we have that the co-diagonal
	$\nabla _{La} \colon La \otimes La \to La $ is closed. By \cref{prop:z_preserves_closedness}, each 
	$\mathfrak{z} (L \sqcup L)(a) \cong \mathfrak{z}(La \otimes La) \to \mathfrak{z}La$ is closed and thus 
	the co-diagonal $\mathfrak{z}L \sqcup \mathfrak{z}L \to L$ is closed, i.e $\mathfrak{z}L$ is Hausdorff. 
\end{proof}
In particular, consider the constant presheaf $\mathbf{2} \colon \ca{C}^{op} \to \Frm$ on the two element frame, as considered in \cref{ex:comuting_subobj_class}.  Each section is Hausdorff, and so $\mf{z}\mathbf{2}$, i.e.\ the subobject classifier $\Omega \in \Frm[ \ca{C}^{op}, \Set ]$, is Hausdorff too. 
\begin{example}
We close with an example that the sections $La$ of an internally Hausdorff frame $L$ need not be Hausdorff. Our thanks go to S.~Henry for pointing this out to us.  Consider the example above of the subobject classifier $\Omega$ for the particular presheaf topos $[ \mathbf{2}^{op}, \Set ]$.  We have just seen that $\Omega$ is a Hausdorff internal frame. But it is easily calculated that $\Omega (1)$ has three elements, and hence  is isomorphic  to the Sierpi\'nski frame, which is obviously not Hausdorff. 
\end{example}

\subsection*{Acknowledgements}
We thank Simon Henry and Christopher Townsend for sharing their expertise with us, and also J\'er\'emie Marqu\`es for the many encouraging conversations.

\smallskip
\noindent
\begin{minipage}{0.4\textwidth}
	\begin{center}
		\includegraphics[height=1.1cm]{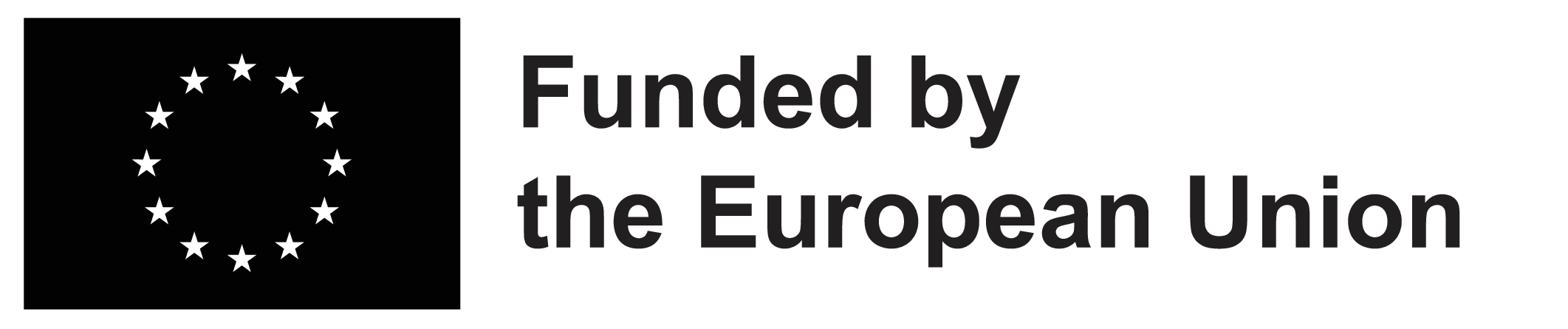}
	\end{center}
\end{minipage}
\begin{minipage}{0.6\textwidth}
	The third author was supported by Marie Sk\l{}odowska-Curie Grant No.\ 101273434.
\end{minipage}

\printbibliography

\end{document}